\documentclass[a4paper,12pt]{article}
\usepackage[margin=1in]{geometry}
\usepackage{amsmath,amsthm,amssymb,amsfonts}
\usepackage{hyperref}
\usepackage{xurl}

\newtheorem{proposition}{Proposition}
\newtheorem{theorem}{Theorem}

\newtheorem{corollary}{Corollary}

\DeclareMathOperator{\Var}{Var}

\title{A direct injection for the strong $q$-log-convexity of Touchard polynomials}
\author{Vuong Bui\thanks{\texttt{bui.vuong@yandex.ru}}}
\date{}

\begin{document}

\maketitle

\begin{abstract}
	We provide a direct injection for the well-known strong log-convexity of the Bell numbers $B_n$, that is $B_mB_n\le B_{m-1}B_{n+1}$ for every $1\le m\le n$. Our injection $\Pi_m\times\Pi_n\to \Pi_{m-1}\times\Pi_{n+1}$, where $\Pi_n$ denotes the set of all partitions of $[n]$, preserves the total number of blocks in the pair of partitions. In other words, it is also an injection for the strong $q$-log-convexity of Touchard polynomials, a result established by Chen, Wang, and Yang using analytical arguments. As an application of the injection, we also recover a related result of Chern, Diaconis, Kane, and Rhoades. 
\end{abstract}

\section{Introduction}

Let $B_n$ denote the Bell numbers, that is the number of partitions of $[n]$. One can check the sequence, numbered $A000110$, on The On-Line Encyclopedia of Integer Sequences \cite{OEIS_A000110} for more information. It is well known that $(B_n)^2\le B_{n-1}B_{n+1}$ for every $n\ge 1$. This so-called log-convexity of $B_n$ was first proved by Engel \cite{engel1994average} in 1994, with a more combinatorial proof later by Canfield \cite{canfield1995engel} in 1995. Although the latter proof uses some combinatorial insight through the number of blocks in each partition, the main step still involves analysis with some elementary inequalities, mostly of Cauchy--Schwarz type. Therefore, we may wonder if we could make a direct injection from a pair of partitions of $[n]$ to a pair of partitions, the first of $[n-1]$ and the second of $[n+1]$. A direct strategy would be to remove the element $n$ in the first partition and add the element $n+1$ to the second partition. However, the delicate part is how to keep track of the block from which the element $n$ is removed and the block to which the element $n+1$ is added. When studying the log-convexity of combinatorial sequences, Liu and Wang \cite{liu2007log} posed the following problem in 2007, after the log-convexity of the Bell numbers had been established by several different methods:
\begin{quote}
An intriguing problem is to find a combinatorial interpretation for the log-convexity of the Bell numbers.
\end{quote}
Sun and Wang subsequently gave a combinatorial proof using weighted Motzkin paths \cite{sun2014combinatorial}. They further observed that it would be interesting to construct a weight-nondecreasing injection between the corresponding pairs of weighted paths. More recently, Chen, Wang, and Zheng \cite{chen2021combinatorial} obtained an injective proof formulated in terms of weighted \L{}ukasiewicz paths. In contrast to these path-based approaches, the present article constructs an explicit injection directly on pairs of set partitions and preserves the total number of blocks across the two partitions.
The problem has also attracted broader interest, as evidenced by a related discussion on Mathematics Stack Exchange.\footnote{\url{https://math.stackexchange.com/questions/67988/an-inequality-involving-bell-numbers-b-n2-leq-b-n-1b-n1}}

The intriguing part may be due to the following question: if we can move an element from one partition to the other, what prevents us from reversing the procedure for a pair of partitions of $[n-1]$ and $[n+1]$? One may argue that what we really work on is the blocks, not the numbers of elements $n-1,n,n+1$. It turns out that, as one can see in the proof later, the reason we cannot do it the other way around is partly the symmetry: we may exchange the order of the two partitions with the same total number of elements $n$ when necessary.\footnote{In fact, when looking at the proof, one can see that the reason becomes more involved when we extend the strategy to cover the strong $q$-log-convexity. The symmetry is our original motivation, however.}

In this article, we assume that partitions are sorted in the increasing order of the smallest element of each block. That is, if $P=\{P_1,\dots,P_s\}$, then $\min P_1\le\dots\le \min P_s$. This order is standard in representing partitions, but it proves to be particularly useful in our case.

While Bell numbers only involve the number of partitions, Touchard polynomials keep track of the number of blocks in each partition. In particular, Touchard polynomials $T_n$ are defined by
\[
	T_n(q) = \sum_{k=0}^n \left\{ {n \atop k} \right\} q^k.
\]
The $q$-log-convexity of Touchard polynomials is the property that
\[
	[T_n(q)]^2\le_q T_{n-1}(q) T_{n+1}(q),
\]
where $P_1(q) \le_q P_2(q)$ for two polynomials $P_1(q),P_2(q)$ means that the coefficients of $P_2(q)-P_1(q)$ are all nonnegative. This property was proved by Liu and Wang \cite{liu2007log}.

By strong log-convexity for Bell numbers, we mean the property that
\[
	B_mB_n\le B_{m-1}B_{n+1}
\]
for any $1\le m\le n$.
Likewise, the strong $q$-log-convexity of Touchard polynomials is defined by
\begin{equation}\label{eq:strong-q}
	T_m(q)T_n(q)\le_q T_{m-1}(q)T_{n+1}(q)
\end{equation}
for any $1\le m\le n$.
This property was proved in \cite{chen2011recurrence} by Chen, Wang, and Yang in 2011.

While the strong log-convexity and the log-convexity for any positive sequence can be shown to be equivalent by simple arithmetic operations, we do not have the equivalence for the strong $q$-log-convexity and the $q$-log-convexity in general \cite{chen2011recurrence} (for a simple counterexample, see \cite{chen2010schur}). 
To prove \eqref{eq:strong-q} in a combinatorial way, it suffices to prove the following version.

\begin{theorem}
\label{thm:strong-q}
	Let $\Pi_n$ denote the set of all partitions of $[n]$. For every $1\le m\le n$, there exists an explicit injection $\Pi_m\times\Pi_n \to \Pi_{m-1}\times \Pi_{n+1}$ so that the total number of blocks in the two partitions is preserved.
\end{theorem}
In order to prove Theorem \ref{thm:strong-q}, we reduce it to a slightly more convenient form in Theorem \ref{thm:marked} as a way to relate it with the proof in \cite{chen2011recurrence}.
In particular, the proof in \cite{chen2011recurrence}, which is more analytical, gives the following intermediate step.\footnote{The formula appears at the end of the proof of \cite[Theorem $2.4$]{chen2011recurrence}, rewritten for Touchard polynomials.}
\begin{proposition}[Chen-Wang-Yang 2011]
\label{prop:quant-strong-q}
	For every $1\le m\le n$,
	\[
		T_{m-1}(q)T_{n+1}(q)-T_m(q)T_n(q) = q(T_{m-1}(q)T_n'(q)-T_{m-1}'(q)T_n(q)).
	\]
\end{proposition}

After that, the proof in \cite{chen2011recurrence} shows that the right hand side has only nonnegative coefficients.

One should note that the sequence for 
\[
	qT'_{n-1}(q) = \sum_{k=0}^{n-1} k\left\{ {n-1 \atop k} \right\} q^k
\]
is almost the same as the one for $T_n(q)$, except that we require the block containing $n$ to contain another element as well. This is due to the multiplier $k$, which corresponds to the number of ways to insert $n$ into a block of a partition of $[n-1]$. We denote by $\Pi'_n$ the set of all such partitions of $[n]$. Also, we denote by $\Pi^*_n$ the set of all partitions with the last block being precisely $\{n\}$ (whose cardinality is $B_{n-1}$). Obviously,
\[
	\Pi_n = \Pi'_n\sqcup \Pi^*_n.
\]

We can prove Proposition \ref{prop:quant-strong-q} quite easily with a bijection.
\begin{proof}[Bijective proof of Proposition \ref{prop:quant-strong-q}]
	It suffices to give a bijection
	\begin{equation}\label{eq:bijection}
		(\Pi_m\times \Pi_n)\sqcup (\Pi_{m-1}\times \Pi'_{n+1}) \longleftrightarrow (\Pi'_{m}\times \Pi_n) \sqcup (\Pi_{m-1}\times \Pi_{n+1})
	\end{equation}
	that preserves the total number of blocks.
	Firstly, we have a simple bijection
	\[
		\Pi^*_m\times \Pi_n \leftrightarrow \Pi_{m-1}\times \Pi^*_{n+1},
	\]
	by removing the block $\{m\}$ from $P$ and appending $\{n+1\}$ to $Q$, for a pair $(P,Q)$ on the left hand side, to obtain a unique pair on the right hand side (the reverse direction can be done likewise). The bijection preserves the total number of blocks.

	Taking the disjoint union of the previous bijection with the following obvious bijection side by side,
	\[
		(\Pi'_m\times \Pi_n)\sqcup (\Pi_{m-1}\times \Pi'_{n+1}) \longleftrightarrow (\Pi'_m\times \Pi_n) \sqcup (\Pi_{m-1}\times \Pi'_{n+1}),
	\]
    we obtain the desired bijection \eqref{eq:bijection},
	where $(\Pi^*_m\times \Pi_n) \sqcup (\Pi'_m\times \Pi_n) = \Pi_m\times \Pi_n$ and $(\Pi_{m-1}\times \Pi^*_{n+1}) \sqcup  (\Pi_{m-1}\times \Pi'_{n+1}) = \Pi_{m-1}\times \Pi_{n+1}$. The conclusion follows.
\end{proof}

Given Proposition \ref{prop:quant-strong-q}, in order to prove the strong $q$-log-convexity of the Touchard polynomials, it remains to prove that its right hand side has only nonnegative coefficients.
\begin{theorem}[Chen-Wang-Yang 2011] \label{thm:nonnegative}
	For every $1\le m\le n$,
	\[
		q(T_{m-1}(q)T_n'(q)-T_{m-1}'(q)T_n(q)) \ge_q 0.
	\]
\end{theorem}

We have the following combinatorial version of Theorem \ref{thm:nonnegative}.

\begin{theorem}
\label{thm:marked}
	For any $1\le m\le n$, there exists an explicit injection $\Pi'_m\times \Pi_n\to \Pi_{m-1}\times \Pi'_{n+1}$ that preserves the total number of blocks in the two partitions.
\end{theorem}

\begin{corollary}
    Theorem \ref{thm:marked} implies Theorem \ref{thm:strong-q}.
\end{corollary}
The corollary can be seen directly from the bijection in the proof of Proposition \ref{prop:quant-strong-q}. We nevertheless give the details, since Theorem \ref{thm:strong-q} is the main result of the paper.
\begin{proof}
    We have the corollary since
    \[
        \Pi_m\times\Pi_n = (\Pi'_m\sqcup \Pi^*_m)\times \Pi_n
    \]
    and
    \[
        \Pi_{m-1}\times\Pi_{n+1} = \Pi_{m-1} \times (\Pi'_{n+1}\sqcup \Pi^*_{n+1}),
    \]
    while there is already a bijection
    \[
        \Pi^*_m\times \Pi_n \leftrightarrow \Pi_{m-1}\times \Pi^*_{n+1}
    \]
    that preserves the total number of blocks (as already mentioned in the proof of Proposition \ref{prop:quant-strong-q}).
    As the two target sets are disjoint, the union of this bijection and the injection in Theorem \ref{thm:marked} is the required injection in Theorem \ref{thm:strong-q}.
\end{proof}

Let us mention a quantitative result on the log-convexity. In particular, if we let $K_n$ be the number of blocks in a uniformly random partition of $[n]$, the following result can be derived from Chern, Diaconis, Kane, and Rhoades \cite[Proposition 1]{chern2014closed}.
\begin{theorem}[Chern–Diaconis–Kane–Rhoades 2014]
\label{thm:quant-log-convex}
	For every $n\ge 1$,
	\[
		B_{n-1}B_{n+1} - (B_n)^2 = (B_{n-1})^2(1 + \Var(K_{n-1})).
	\]
\end{theorem}
We also mention that analytical techniques in \cite{alzer2019engel} allow us to present the difference $B_{n-1}B_{n+1} - (B_n)^2$ as infinite series with nonnegative terms.
In this article, our combinatorial approach recovers Theorem \ref{thm:quant-log-convex}.

In the following sections, we prove Theorem \ref{thm:marked} (which implies Theorem \ref{thm:strong-q}) and Theorem \ref{thm:quant-log-convex}.

\section{Proof of Theorem \ref{thm:marked}}
For each pair of partitions $(P,Q)$ in $\Pi'_m\times \Pi_n$, where $1\le m\le n$, we assign a unique pair of partitions $(P',Q')$ in $\Pi_{m-1}\times \Pi'_{n+1}$. We denote the blocks of $P,Q,P',Q'$ by $P=\{P_1,\dots,P_s\}$, $Q=\{Q_1,\dots,Q_t\}$, $P'=\{P'_1,\dots,P'_{s'}\}$ and $Q'=\{Q'_1,\dots,Q'_{t'}\}$. 

During the operations below, we do not change the order of pre-existing blocks, except that we may move some blocks from the end of a partition to the end of the other partition. The order is unchanged because no smallest element of a block gets removed, and no smaller element gets inserted. This fact also implies that the total number of blocks in the two partitions is preserved.

Let $i$ be the index of the block containing $m$ in $P$. We consider the following two cases: 
\begin{itemize}
	\item $i\le t$: 
		We remove $m$ from $P_i$ to obtain $P'$, and add $n+1$ to $Q_i$ to obtain $Q'$.

		Recovering $P,Q$ from a resulting pair $P',Q'$ is immediate: remove $n+1$ from $Q'_j$ and add $m$ to $P'_j$. The feature distinguishing this case from the other case is that the block $Q'_j$ containing $n+1$ satisfies $j\le s'$. 
        
	\item $i>t$: We rename the element $m$ in $P_i$ to $n+1$. After that, we move elements $m,\dots,n$ from $Q$ to $P$ as follows. (Note that in the end we exchange $P$ and $Q$.) For each $x\in \{m,\dots,n\}$, if the block $Q_j$ containing $x$ also contains an element other than $m,\dots,n$, we move $x$ from $Q_j$ to $P_j$. Since $m,\dots,n$ are the largest elements of $Q$, the remaining elements of $m,\dots,n$ (which have not been moved) constitute contiguous blocks at the end of $Q$. We move these blocks from $Q$ to the end of $P$. Now we exchange their roles and set $P'=Q$ and $Q'=P$. 

    Recovering $P,Q$ from a resulting pair $P',Q'$ is straightforward. For each $r\le s'$, we move every element of $Q'_r\cap\{m,\dots,n\}$ back to $P'_r$. After that, we identify the maximal sequence of blocks at the end of $Q'$ that are contained in $\{m,\dots,n\}$ and move them back to the end of $P'$. These blocks are uniquely identifiable, since every block originally coming from $P$ contains an element smaller than $m$. We then rename $n+1$ in $Q'$ to $m$. Finally, we swap their roles to obtain $P=Q'$ and $Q=P'$. 

		The feature distinguishing this case from the other case is that the block $Q'_j$ containing $n+1$ satisfies $j > s'$.
\end{itemize}

Thus the two cases have disjoint images, and the recovery procedures show that the map is injective.
This concludes the proof of Theorem \ref{thm:marked}.

\section{Proof of Theorem \ref{thm:quant-log-convex}}
This section continues the proof of Theorem \ref{thm:marked} and uses the notation from there.

Now we consider the case $m=n$ to prove Theorem \ref{thm:quant-log-convex}.

In the first case $i\le t$, the set of all pairs $(P',Q')$ so that the block $Q'_j$ containing $n+1$ satisfies $j\le s'$ is the image of the injection. Indeed, given any such pair $(P',Q')$, if the block $Q'_j$ contains $n+1$, then we remove $n+1$ from $Q'_j$ to obtain $Q$ and add $n$ to $P'_j$ to obtain $P$.

Unlike the first case, in the case $i>t$, the image of the injection is only a subset of all pairs $(P',Q')$ satisfying the condition that the block $Q'_j$ containing $n+1$ has $j > s'$. In particular, let $Q'_k$ be the block containing $n$. For a pair $(P',Q')$ in the image, we have either $k\le s'$, or $k=t'$ with $Q'_{t'}=\{n\}$. While the former case is clear, in the latter case $k=t'$, we need $j>s'+1$ as well. Indeed, in this case the original $Q$ has $t=s'+1$ blocks (as one block gets moved). Meanwhile, we have $i>t$ and $j=i$. Therefore, $j>s'+1$.

We show that either $k\le s'$, or $k=t'$ with $Q'_{t'}=\{n\}$ and $j>s'+1$, is also sufficient by recovering $P,Q$. 
If $k\le s'$, then we move $n$ from $Q'_k$ to $P'_k$. If $k=t'$ and $j>s'+1$, we move the block $\{n\}$ from the end of $Q'$ and append it to the end of $P'$. Finally, we rename $n+1$ in $Q'$ to $n$. We now exchange $P',Q'$ to obtain $P,Q$.

Therefore, the pairs $(P',Q')$ that are not covered by the injection are those with either $s'<k\le t'$ with $Q'_k\ne\{n\}$, or $k=t'$ with $Q'_{t'}=\{n\}$ and $j=s'+1$. To count them, we start with two partitions $A=\{A_1,\dots,A_u\}$ and $B=\{B_1,\dots,B_v\}$ in $\Pi_{n-1}$ so that $u\le v$. There are $\left\{{n-1\atop u}\right\}\left\{{n-1\atop v}\right\}$ choices. After that, we will add $n,n+1$ into $B$. ($A,B$ will play the role of $P',Q'$ later.) We count these pairs according to the following cases:
\begin{itemize}
	\item $s'<k\le t'$ with $Q'_k\ne\{n\}$ (and $j>s'$): If we do not add any new block, there are $v-u$ blocks to insert $n$ into, and the same $v-u$ blocks to insert $n+1$. If we add a new block, the only choice is the block $\{n,n+1\}$, as neither $n$ nor $n+1$ could stay in a singleton block.
	\item $k=t'$ with $Q'_{t'}=\{n\}$ and $j=s'+1$: This works only when $u<v$, for which there is only one choice of appending $\{n\}$ to $B$ and inserting $n+1$ into $B_{u+1}$.
\end{itemize}

Therefore, the difference $B_{n-1}B_{n+1} - (B_n)^2$ is
\[
	\sum_{0\le u\le v\le n-1} \left\{{n-1 \atop u}\right\} \left\{{n-1 \atop v}\right\} \left((u-v)^2 + 1 + [v>u]\right),
\]
where $[v>u]$ is $1$ if $v>u$ and $0$ otherwise.
It can be rewritten in a nicer way as
\[
	\frac{1}{2} \sum_{u=0}^{n-1} \sum_{v=0}^{n-1} \left\{{n-1 \atop u}\right\} \left\{{n-1 \atop v}\right\} [(u-v)^2 + 2] = (B_{n-1})^2(1 + \Var(K_{n-1})).
\]
The conclusion follows.
\bibliographystyle{unsrt}
\bibliography{touchard}
\end{document}